\documentclass[twoside,11pt,leqno]{article}

\usepackage{amsthm, amssymb, latexsym, amsmath, amscd, array, graphicx, enumerate, lmodern, slantsc, hyperref,amsfonts}
\usepackage[english]{babel}
\usepackage[all]{xy}
\usepackage{mathtools}

\CompileMatrices

\usepackage{color}
\def\red{} 
\def\redd{} 
\def\blue{} 

\def\TC{\protect\operatorname{TC}}
\def\cd{\protect\operatorname{cd}}
\def\cat{\protect\operatorname{cat}}

\def\hdim{\protect\operatorname{hdim}}

\def\max{\protect\operatorname{max}}

\newtheorem{proposition}{Proposition}[section]
\newtheorem{definition}[proposition]{Definition}
\newtheorem{theo}[proposition]{Theorem}
\newtheorem{remark}[proposition]{Remark}

\newtheorem{lema}[proposition]{Lemma}

\newtheorem{corollary}[proposition]{Corollary}

\makeindex

\title{\red{Pairwise disjoint maximal cliques in random graphs and sequential motion planning on random right angled Artin groups}}
\author{Jes\'us Gonz\'alez\thanks{\red{Partially supported by Conacyt Research Grant 221221.}}, B\'arbara Guti\'errez, and Hugo Mas}

\date{\today}

\begin{document}

\maketitle

\begin{abstract}
\red{The clique number of a random graph in the Erd\"os-R\'enyi model $\mathcal{G}(n,p)$ yields a random variable which is known to be asymptotically (as $n$ tends to infinity) almost surely within one of an explicit logarithmic (on $n$) function $r(n,p)$. We extend this fact by showing that random graphs have, asymptotically almost surely, arbitrarily many pairwise disjoint complete subgraphs with as many vertices as $r(n,p)$. The result is motivated by and applied to the sequential motion planning problem on random right angled Artin groups. Indeed, we give an asymptotical description of all the higher topological complexities of Eilenberg-MacLane spaces associated to random graph groups.}
\end{abstract}

\medskip
\noindent{\red{{\it 2010 Mathematics Subject Classification}: 60C05, 05C80, 20F36, 52B70, 55M30, 55U10, 68T40.}}

\noindent{\red{{\it Keywords and phrases:} Random graph, Erd\"os-Renyi model, maximal cliques, Schwarz genus, Sequential motion planning, polyhedral product, right angled Artin group.}}

\tableofcontents

\section{Introduction}\label{secintro}
\red{For a positive integer $n$ and probability parameter $p$, $0<p<1$, consider the Erd\"os-R\'enyi model $\mathcal{G}(n,p)$ of random graphs $\Gamma$ in which each edge of the complete graph on the $n$ vertices $[n]=\{1,2,\ldots, n\}$ is included in $\Gamma$ with probability $p$ independently of all other edges. In other words, the random variables $e_{ij}$, $1\leq i<j\leq n$, defined by
$$
e_{i,j}(\Gamma)=\begin{cases}
1, & \mbox{if $(i,j)$ is an edge in $\Gamma$;}\\
0, & \mbox{otherwise,}
\end{cases}
$$
are independent and have $P(e_{i,j}=1)=p$. In this context, the clique random variable $C=C_{n,p}$, $$C(\Gamma)=\max\{r\in\mathbb{N}\;\colon \Gamma\mbox{ admits a complete subgraph with $r$ vertices}\},$$ has been the subject of intensive research since the 1970's. Matula provided in~\cite{Matula1} numerical evidence suggesting that $C$ has a very peaked density around $2\log_{q}n$ where $q=1/p$. Such a property was established in~\cite{gridia} by Grimmett and McDiarmid who proved that, as $n\to\infty$, $$\frac{C}{\log_qn}\to2.$$ A much finer result, Theorem~\ref{randomcat} below, was proved by Matula. Here and throughout the paper $\lfloor x\rfloor$ stands for the integral part of the real number $x$, and we set $$z=z(n,p)=2\log_qn-2\log_q\log_qn+2\log_q(e/2)+1$$ where, as above, $q=1/p$.}

\begin{theo}[{\cite[Equation~(2)]{Matula2}}]\label{randomcat}
\red{For $0<p<1$ and $\epsilon>0$,} $$\red{\lim_{n\to\infty}\mathrm{Prob}\left(\lfloor z-\epsilon\rfloor\leq C\leq\lfloor z+\epsilon\rfloor\rule{0mm}{4mm}\right)=1.}$$
\end{theo}

\red{It should be stressed that the probability parameter $p$ is fixed throughout the limiting process. In common parlance, Theorem~\ref{randomcat} can be stated by the assertion that, for a fixed $p\in(0,1)$, the inequalities $\lfloor z-\epsilon\rfloor\leq C(\Gamma)\leq\lfloor z+\epsilon\rfloor$ hold {\it asymptotically almost surely} for random graphs $\Gamma\in\mathcal{G}(n,p)$. Alternatively, since $0\leq\lfloor z+\epsilon\rfloor-\lfloor z-\epsilon\rfloor\leq 1$ when $\epsilon\leq1/2$, $C$ is asymptotically almost surely determined by $z$ with spikes of at most a unit whose appearance depend on the ``resolution'' parameter $\epsilon$ used.}

\medskip
\red{Since $z$ is logarithmic in $n$, it is conceivable that Theorem~\ref{randomcat} admits a strengthening to an assertion about the existence in random graphs of arbitrarily many pairwise disjoint asymptotically-largest-possible cliques. A first step in such a direction was taken in~\cite{CosFar} in response to the desire of understanding the stochastic properties of the collision-free motion planning of multiple particles on graphs with a large number of vertices. Roughly speaking, Costa and Farber showed that, with probability tending to~$1$ as $n$ tends to infinity, a random graph in $\mathcal{G}(n,p)$ has a pair of disjoint asymptotically-largest-possible cliques. In our first main result (Theorem~\ref{randomTCs} below) we show, more generally, that for any fixed positive integer $s$, and with probability tending to 1 as $n$ tends to infinity, a random graph in $\mathcal{G}(n,p)$ has $s$ pairwise-disjoint such asymptotically-largest-possible cliques.}

\begin{definition}\label{multiclique}
\red{An $s$-th multi-clique of size $r$ of a (random) graph $\Gamma\in\mathcal{G}(n,p)$ is an ordered $s$-tuple $(V_1,\ldots,V_s)$ of pairwise disjoint subsets $V_i\subseteq [n]$, each of cardinality $r$, such that each of the induced  subgraphs $\Gamma_{|V_i}$ is complete.}
\end{definition}

\red{We do not require in Definition~\ref{multiclique} that each $V_i$ is a clique of $\Gamma$ (i.e.~a complete subgraph of $\Gamma$ with the maximal possible number of vertices). Yet the word \emph{multi-clique} has been used in the above definition in view of Theorem~\ref{randomcat}, and since we will be concerned with the case $r=\lfloor z-\epsilon\rfloor$ for some (small but fixed) $\epsilon>0$.}

\medskip
\red{Section~\ref{seccliqes} is devoted to the proof of:}

\begin{theo}\label{randomTCs}
\red{Fix a positive integer $s$, a positive real number $\epsilon$, and a probability parameter $p\in(0,1)$. Then, with probability tending to 1 as $n\to\infty$, a random graph in $\mathcal{G}(n,p)$ has an $s$-th  multi-clique of size $\lfloor z-\epsilon\rfloor$.}
\end{theo}

\red{We use Theorem~\ref{randomTCs} in order to generalize Costa and Farber's result to the sequential motion planning realm in topological robotics. In Section~\ref{sectcs} we review the definition and basic properties of Rudyak's higher topological complexity, a concept generalizing Farber's topological viewpoint of the motion planning problem. We then use Theorem~\ref{randomTCs} to give the following asymptotical description (with $\epsilon$-resolution spikes of at most $s$ units) of the $s$-th topological complexity of random right angled Artin groups:}

\begin{theo}\label{randomtopTCs}
\red{For a random graph $\Gamma\in\mathcal{G}(n,p)$, let $K_\Gamma$ stand for the (random) Eilenberg-MacLane space associated to the right angled Artin group defined by $\Gamma$. Then, for any positive real constant $\epsilon$, positive integer $s$, and probability parameter $p\in(0,1)$, the random variable $\TC_s$ given by $\TC_s(\Gamma)=\TC_s(K_\Gamma)$ satisfies}
$$
\red{\lim_{n\to\infty}\mathrm{Prob}\left(s\lfloor z-\epsilon\rfloor\leq\TC_s \leq s\lfloor z+\epsilon\rfloor\rule{0mm}{4mm}\right)=1.}
$$
\end{theo}

\medskip
\red{As explained in Proposition~\ref{desccoho} (see also~\cite[Section~3]{CosFar}), in this context it is safe to set $\TC_1(K_\Gamma)=\cat(K_\Gamma)$, the Lusternik-Shnirelmann category of $K_\Gamma$. The case $s=1$ in Theorem~\ref{randomtopTCs} then gives a topological interpretation of Matulas's Theorem~\ref{randomcat}, while the case $s=2$ in Theorem~\ref{randomtopTCs} recovers the main result in~\cite{CosFar}.}

\medskip
\red{Since the $s$-th topological complexity of a cellular complex $X$ is bounded from above by $s$ times the dimension of $X$ (\cite[Theorem~3.9]{bgrt}), Theorem~\ref{randomtopTCs} implies that the $s$-th topological complexity of random Eilenberg-MacLane spaces for right angled Artin groups asymptotically almost surely lie on an $s$-neighborhood of the largest dimensionally-allowed value (Corollary~\ref{neighborhood} below). \blue{We believe that the neighborhood radius can be dropped down to zero by letting the probability parameter be a suitably chosen function of $n$---this will be worked out elsewhere.}}

\medskip\red{The methods in this paper owe much to the ideas in~\cite{CosFar}. In particular, the proof of Theorem~\ref{randomtopTCs} is based on the fact that $K_\Gamma=(S^1,\star)^{\Delta_\Gamma}$, the polyhedral product complex determined by the (based) unit circle $(S^1,\star)$ and by the flag (clique) complex $\Delta_\Gamma$ spanned by $\Gamma$. Using the full strength of~\cite{GGY}, where a combinatorial description of the higher topological complexities of more general polyhedral product complexes is given, it should be possible to describe stochastic properties of the higher topological complexities of (random) polyhedral product spaces associated to other families of (random) abstract simplicial complexes.} 

\section{Maximal disjoint cliques}\label{seccliqes}
\red{We now deal with the proof of Theorem~\ref{randomTCs}. Throughout this section we let $r:=\lfloor z-\epsilon\rfloor$, a function on $n$, $p$, and $\epsilon$. Although $n$ will indeed vary, in what follows the parameters $p$ and $\epsilon$ (as well as $s$) will be kept fixed. We will assume $s\geq2$, as the case $s=1$ in Theorem~\ref{randomTCs} is covered by Theorem~\ref{randomcat}.}



\medskip
Let $X_{r,s}: \mathcal{G}(n, p) \rightarrow \mathbb{Z}$ be the random variable that assigns to each random graph the number of its $s$-th multi-cliques of size $r$. \red{Note that $X_{r,s}(\Gamma)$ is divisible by $s!$, for an $s$-th multi-clique is an \emph{ordered} $s$-tuple of disjoint sets. We could of course normalize by dividing by $s!$, but the unnormalized setting yields slightly simpler formulas in the arguments below.}

\medskip
By the second moment method,
\begin{equation}\label{ratio}
\mathrm{Prob}\left(\rule{0mm}{3.5mm}X_{r,s}>0\right)\geq \frac{E(X_{r,s})^2}{E(X_{r,s}^2)},
\end{equation}
\red{so it suffices to show that the ratio on the right hand side of~(\ref{ratio}) tends to 1 as $n\to\infty$.} 

\medskip
\red{Let $\mathcal{W}(s)$ stand for the set of $s$-tuples $(W_1, \ldots, W_s)$ of \red{pairwise} disjoint subsets $W_i$ of $[n]$, each having cardinality~$r$. Each $\mathbf{W}\in\mathcal{W}(s)$ determines a random variable $I_{\mathbf{W}}: \mathcal{G}(n, p) \rightarrow \{0, 1\}$ given by}
$$
\red{I_{\mathbf{W}}(\Gamma)=\begin{cases}
1, & \mbox{if $\mathbf{W}$ is an $s$-th multi-clique of size $r$ of $\Gamma$;}\\
0, & \mbox{otherwise.}
\end{cases}}
$$
In these terms, $X_{r,s}$ can be written as $$X_{r,s}= \sum_{\mathbf{W}\in\mathcal{W}(s)} I_{\mathbf{W}},$$ and since $E(I_{\mathbf{W}})=p^{s\binom{r}{2}}$ for each $\mathbf{W}\in\mathcal{W}(s)$, linearity of the expectation yields
$$
E(X_{r,s})=\displaystyle{\binom{n}{\,\underbrace{r,\ldots,r}_s\,}p^{s\binom{r}{2}}}
$$
where $\binom{a}{\,b_1,\ldots,b_k\,}$ stands for the multinomial coefficient
$$
\binom{a}{b_1,\ldots,b_k}=\frac{a!}{\left(\prod_{i=1}^k b_i ! \right) \left(a-\sum_{i=1}^k b_i \right) !}
$$
determined by non-negative integers $a,b_1,\ldots,b_k$ with $k\in\mathbb{N}$ and $a\geq\sum_{i=1}^kb_i$. On the other hand, in order to deal with $E(X^2_{r,s})$, write $X_{r,s}^2=\sum I_{\mathbf{W}}\cdot I_{\mathbf{W}'}$ and note that
$$
E(I_{(W_1,\ldots,W_s)}\cdot I_{(W_1',\ldots,W_s')})= p^{2s \binom{r}{2}- \sum \binom{a_{ij}}{2}}
$$
where we set $a_{ij}:=|W_i \cap W'_j|$. \red{We say that the pair $(W,W')$ has intersection type given by the matrix $A=(a_{ij})$.}

\medskip
Before using the previous considerations to estimate the right hand side term of~(\ref{ratio}), it is convenient to introduce some auxiliary notation. Given an $(s \times s)$-matrix $A=(a_{ij})$ with integer coefficients, let $A_i$ and $A^i$ ($1 \leq i \leq s$) denote the $s$-th tuples determined by the $i$-th row and the $i$-th column of $A$, respectively. Moreover, let $\Sigma(c_1,\ldots,c_s):=\Sigma_{i=1}^{s}c_i$.
In these terms, we get
\begin{equation}\label{expec2}
\frac{E(X^2_{r,s})}{E(X_{r,s})^2}= \sum_{A \in D} F_A \cdot q^{L(A)}= \sum_{A \in D} T_A.
\end{equation}
Here the summations run over the set $D$ of $(s\times s)$-matrices $A=(a_{ij})$ with non-negative integer coefficients satisfying $\red{\max_{1\leq i\leq s}\left\{\Sigma A_i,\Sigma A^i\right\}\leq r,}$ and we have set
$$
F_A=\frac{\binom{r}{\,A^1\,}\binom{r}{\,A^2\,}\cdots \binom{r}{\,A^s\,}\binom{n-sr}{\,r-\Sigma A_1,\,r-\Sigma A_2,\ldots,\,r-\Sigma A_s\,}}{\binom{n}{\,\underbrace{\mbox{\scriptsize$r,\ldots,r$}}_s\,}},\;\;\, L(A)=\sum_{i, j=1}^s \binom{a_{ij}}{2},\;\;\,\mbox{and}\;\;\;T_A=F_A\cdot q^{L(A)}.$$
\begin{remark}\label{sumfa}{\em
Note that
$$\sum_{A \in D} F_A =\red{\frac{\sum_{A \in D}\binom{n}{\,r,\ldots,r\,}\binom{r}{\,A^1\,}\binom{r}{\,A^2\,}\cdots \binom{r}{\,A^s\,}\binom{n-sr}{\,r-\Sigma A_1,\,r-\Sigma A_2,\ldots,\,r-\Sigma A_s\,}}{\binom{n}{\,r,\ldots,r\,}^2}}=1,$$
\red{as both the numerator and denominator in the quotient give the cardinality of $\mathcal{W}(s)^2$.}
}\end{remark}

Our updated task is to show that $\sum_{A\in D} T_A\to1$ as $n\to\infty$. In fact, Lemma~\ref{auxil} below implies that it suffices to show
\begin{equation}\label{sumTA}
\lim_{n\to\infty}\left(\,\sum_{A\in D-\{A_0\}} T_A\right)=0
\end{equation}
where $A_0 \in D$ is the $0$-matrix.

\begin{lema}\label{auxil}
$\lim_{n\to\infty}T_{A_0}=1$.
\end{lema}
\begin{proof} We have

\vspace{-8mm}
\begin{eqnarray*}
T_{A_0}&=& F_{A_0}\;\;=\;\;\frac{\binom{n-rs}{\,\underbrace{\mbox{\scriptsize$r,\ldots,r$}}_s\,}}{\binom{n}{\,\underbrace{\mbox{\scriptsize$r,\ldots,r$}}_s\,}}\;\;=\;\;\frac{\frac{(n-rs)!}{(r!)^s(n-2rs)!}}{\frac{n!}{(r!)^s(n-rs)!}}\\
&=&\frac{(n-rs)!(n-rs)!}{n!(n-2rs)!}\;\;=\;\;\frac{(n-2rs+1)\cdots (n-rs)}{(n-rs+1)\cdots n}\\
&=&\prod_{k=0}^{sr-1}\left(\frac{n-k-sr}{n-k}\right)\;\;=\;\;\prod_{k=0}^{sr-1}\left(1-\frac{sr}{n-k}\right)\\&\geq&\left(1-\frac{sr}{n-sr+1} \right)^{sr}\,=\;\; \left(\left(1-\frac{sr}{n-sr+1} \right)^{2r} \right)^{s/2}.
\end{eqnarray*}
\red{Further, since $\left(1-\frac{sr}{n-sr+1} \right)^{2r}$ can be written as
\begin{eqnarray*}
\left(\hspace{-.3mm}1-\frac{2sr^2}{n-sr+1}\right)+\left[\!\binom{2r}{2}\!\!\left(\frac{sr}{n-sr+1}\right)^{\!2}-\binom{2r}{3}\!\!\left(\frac{sr}{n-sr+1}\right)^{\!3}\right]+\cdots+\left[\!\left(\frac{sr}{n-sr+1}\right)^{\!2r}\right],
\end{eqnarray*}
we see that}
$$
\red{T_{A_0}\geq\left(1-\frac{2sr^2}{n-sr+1} \right)^{s/2}}
$$
for $n$ large enough\footnote{\red{Here and in what follows we use without further notice the easily checked fact that $r^k=o(n)$ for any positive integer $k$.}}. The result then follows from Remark~\ref{sumfa} and from the fact that the term on the right hand side of the latter inequality tends to 1 as $n\to\infty$.
\end{proof}

\redd{The rest of this section is devoted to the proof of~(\ref{sumTA}), which} requires a number of technical preliminary results.  Our first goal is Proposition~\ref{mainine} below, a generalization of~\cite[Lemma~6]{CosFar}.
\begin{lema}\label{c_n}
\redd{For each} positive integer $m$, \redd{there is a positive integer $N(m)$ and a} positive real number~\redd{$\alpha(m)$} such that the number $c_n$ defined through the formula
$$
\binom{n}{mr}=c_n\left(\frac{n}{mr}\right)^{mr}e^{mr}(mr)^{-1/2}
$$
satisfies $c_n \geq \redd{\alpha(m)} > 0$ \redd{whenever $n\geq N(m)$.} 
\end{lema}
\begin{proof}
Using Stirling's formula for factorials (see for instance formula~$(1.4)$ in \cite{Bollobas})
\begin{equation}\label{stirlingformula}
n != \Big(\frac{n}{e} \Big)^n \sqrt{2 \pi n} \hspace{1mm}e^{\alpha_n}, \quad \quad \frac{1}{12n+1} < \alpha_n < \frac{1}{12n},
\end{equation}
 we have
\begin{eqnarray*}
\binom{n}{mr}&=& \frac{1}{\sqrt{2\pi}} \Big(\frac{n}{mr}\Big)^{mr}\Big(\frac{n}{n-mr}\Big)^{n-mr}\sqrt{\frac{n}{mr(n-mr)}}\;\; \ell_n  \\
& =&  c_n\left( \frac{n}{mr}\right)^{mr} e^{mr}{(mr)}^{-\frac{1}{2}}
\end{eqnarray*}
where 
$$\ell_n=\frac{e^{\alpha_n}}{e^{\alpha_{\redd{m}r}} e^{\alpha_{n-\redd{m}r}}} \rightarrow  1\quad \text{ as} \quad n \rightarrow \infty$$
and
$$c_n=\frac{1}{\sqrt{2\pi}} \left( \frac{n}{n-mr}\right)^{\!n-mr}\!\sqrt{\frac{n}{n-mr}}\;\,e^{-mr}\; \ell_n.$$ In order to check that, for large enough $n$, $c_n$ is bounded from below by a fixed positive real number $\alpha$ \redd{(which in general depends on $m$),} we use the inequality $$\left(\frac{a-b}{a}\right)^x \leq e^{-\frac{b}{a} x},$$ which holds for any \red{positive integers $a$, $b$, and $x$ with $b < a$.} Taking in particular $a=n$, $b=mr$ and $x=n-mr-1$, we get
$$
\left( \frac{n-mr}{n} \right)^{n-mr-1} \leq  e^{- \frac{mr}{n}(n-mr-1)} = e^{-mr} e^{\frac{m^2r^2}{n}} e^{\frac{mr}{n}}
$$
or, equivalently,
$$
\left( \frac{n-mr}{n} \right)^{-1} e^{-\frac{m^2r^2}{n}} e^{-\frac{mr}{n}} \leq  \left(\frac{n}{n-mr} \right)^{n-mr}e^{-mr}.
$$
Since the left hand side of the latter inequality approaches $1$ as $n\to\infty$, there exits a positive real number $\alpha$ such that $$c_n=\frac{1}{\sqrt{2\pi}} \left( \frac{n}{n-mr}\right)^{\!n-mr}\sqrt{\frac{n}{n-mr}}\;\,e^{-mr} \;\ell_n > \alpha > 0,$$ for $n$ large enough.
\end{proof}

\begin{proposition}\label{mainine}
Fix non-negative integers
$k$ and $m$ with $m>0$. Then
$$
\lim_{n\to\infty}\left(r^{-k}\binom{n}{\,\underbrace{r,\ldots,r}_m\,}p^{m\binom{r}{2}}\right)= \infty.
$$
\end{proposition}
\begin{proof}
Recall $r=\lfloor z-\epsilon\rfloor\leq z-\epsilon$, so
\begin{center}
$\begin{array}{ll}
p^{m\binom{r}{2}}\;\;\geq\;\; \left(p^{\frac{z-\epsilon-1}{2}}\right)^{mr}&=\;\;\left(p^{\log_qn-\log_q\log_qn+\log_q(e/2)-\frac{\epsilon}{2}}\right)^{mr}\\
&=\;\;\Big( \frac{2C \log_qn}{en}\Big)^{mr}
\end{array}$
\end{center}
where $C=q^{\frac{\epsilon}{2}} >1.$ Note also that 
$$
\binom{n}{\underbrace{r,\ldots,r}_m}=\frac{n!}{r!^m(n-mr)!}=\binom{n}{mr}\frac{(mr)!}{(r!)^m}.
$$
By Lemma \ref{c_n}, \redd{there is a positive real number $\alpha(m)$ and a large positive integer $N(m)$ so that} 
$$
\binom{n}{mr}=c_n\left(\frac{n}{mr}\right)^{mr}e^{mr}(mr)^{-1/2}
$$ 
\red{holds with $c_n\geq \redd{\alpha(m)}>0$ for \redd{$n\geq N(m)$.}} \redd{For such large values of $n$ we then have}
\begin{eqnarray}\label{f2}
r^{-k}\binom{n}{\underbrace{r,\ldots,r}_m}p^{m\binom{r}{2}} &\geq & r^{-km}c_n \left(\frac{n}{mr}\right)^{mr}e^{mr}(mr)^{-1/2}\frac{(mr)!}{r!^m}\left(\frac{2C\log_q n}{en}\right)^{mr}\nonumber\\
&=&m^{-1/2}r^{-km}r^{-\frac{1}{2}}c_n \frac{(mr)!}{r!^mm^{mr}}\left(\frac{2C\log_q n}{r}\right)^{mr}.
\end{eqnarray}
Using Stirling's formula~(\ref{stirlingformula}), we get
$$
\frac{(mr)!}{r!^m m^{mr}}=\frac{\sqrt{2\pi mr\,}\left( \frac{mr}{e}\right)^{mr}}{\sqrt{2\pi r\,}^{\,m} \left(\frac{r}{e}\right)^{mr}m^{mr}}\,d_n=\frac{\sqrt{2\pi mr}\,}{\sqrt{2\pi r\,}^{\,m}}\,d_n
$$
where $d_n =\redd{e^{\alpha_{mr}}/e^{m\alpha_r}} \rightarrow 1$ as $n \rightarrow \infty.$ Therefore, we can rewrite~(\ref{f2}) as
\begin{eqnarray*}
r^{-k}\binom{n}{\underbrace{r,\ldots,r}_m}p^{m\binom{r}{2}}& \geq & m^{-1/2}r^{-km}r^{-\frac{1}{2}}\frac{\sqrt{2\pi mr\,}}{\sqrt{2\pi r\,}^{\,m} }\left(\frac{2C\log_q n}{r}\right)^{mr}c_n d_n \\
&=&\red{\left(2\pi\right)}^{\,\frac{1-m}{2}}r^{-km}r^{-\frac{m}{2}}\left(\frac{2C\log_q n}{r}\right)^{mr}c_n d_n \\
&=&\red{\left(2\pi\right)}^{\frac{\,1-m}{2}}\left[r^{-\frac{2k+1}{2}}\left(\frac{2C\log_q n}{r}\right)^{r}\right ]^mc_n d_n\;\geq\;\red{\left(2\pi\right)^{\frac{\,1-m}{2}}}\left[r^{-\frac{2k+1}{2}}C^{r}\right ]^mc_n d_n
\end{eqnarray*}
\redd{for $n\geq N(m)$,} where \red{the last inequality holds for $n$ large enough \redd{(condition that can be incorporated by increasing $N(m)$ if needed)} in view of the definition of $r$.} The proof is complete in view of the noted characterization of the sequences $\{c_n\}$ and $\{d_n\}$, and since $r^{-\frac{2k+1}{2}}C^r$ tends to infinity as $n \rightarrow \infty$, for
$$
\log_q(r^{-\frac{2k+1}{2}}C^{r})=\frac{\epsilon}{2}r-\left(\frac{2k+1}{2}\right)\log_q r =\frac{\epsilon}{2}r-\left(\frac{2k+1}{\red{2 \ln q}}\right)\red{\ln r}
$$
tends to infinity as $n\to\infty$.
\end{proof}

Next step toward the proof of~(\ref{sumTA}) is an analysis of the asymptotic behavior of $T_A$ for certain matrices $A \in D-\{A_0\}$. \red{In more detail, recall that the set $D$ depends on~$n$. Using subindices to stress the dependence, we have $D_1\subseteq D_2\subseteq D_3\subseteq\cdots$. In Proposition~\ref{matrix1r} below we will be concerned with sequences of matrices $\{A_n\in D_n\}_{n\geq1}$ whose only non-zero entries lie on the main diagonal and are either 1 or $r$. Such a sequence $\{A_n\in D_n\}_{n\geq1}$ as above will simply be referred to as a \emph{diagonal sequence} and, by abuse of notation, will be denoted by $A\in D$. In addition, by a diagonal sequence $A\in D-\{A_0\}$ we mean one for which no $A_n$ is the zero matrix.}

\begin{proposition}\label{matrix1r}
Any diagonal sequence $A\in \red{D-\{A_0\}}$ satisfies $Q(r)\hspace{.4mm}T_A=o(1)$ for any polynomial $Q$ with real coefficients.
\end{proposition}
\begin{proof}
For each $n\geq1$, let $m=m(n)$ and $m'=m'(n)$ be the integers in $\{0,1,\ldots,s\}$ such that $A_n$ has $m$ entries with value $r$ and $m'$ entries with value $1$ (all of these in the main diagonal of $A_n$). In these terms, the generic $T_A$ is
\begin{eqnarray*}
T_A&=&\frac{r^{m'}\binom{n-sr}{\ \hspace{-2.2mm}\underbrace{\mbox{\scriptsize$r,\ldots,r\,$}}_{s-(m+m')},\,\underbrace{\mbox{\scriptsize $r-1,\ldots,r-1\,$}}_{m'}}}{\binom{n}{\underbrace{\mbox{\scriptsize $\,r,\ldots,r\,$}}_s}}\, q^{m\binom{r}{2}} 
\;\;=\;\;\frac{r^{m'}(n-sr)!r!^s(n-sr)!q^{m\binom{r}{2}}}{n!r!^{s-(m+m')}(r-1)!^{m'}(n-2sr+mr +m')!}\\
&=&\frac{r^{2m'}(n-sr)!r!^s(n-sr)!q^{m\binom{r}{2}}}{n!r!^{s-m}(n-2sr+mr +m')!}\;\;=\;\;\frac{r^{2m'}(n-sr)!r!^m(n-sr)!q^{m\binom{r}{2}}}{n!(n-2sr+mr +m')!}\\
&=&\left[\binom{n}{\underbrace{r,\ldots,r}_m}p^{m\binom{r}{2}}\right]^{-1}\!\!\!\frac{r^{2m'}(n-sr)!(n-sr)!}{(n-mr)!(n-2sr+mr +m')!}.
\end{eqnarray*}
\red{The multinomial coefficient $\binom{n-sr}{\,r,\ldots,r,r-1,\ldots,r-1\hspace{.3mm}}$ in the first line of the above equalities should be ignored if $m=s$ and $m'=0$. Likewise, the multinomial coefficient $\binom{n}{r,\ldots,r}$ in the last line of the above equalities should be ignored if $m=0$.} Note that
\begin{eqnarray*}
\frac{r^{2m'}(n-sr)!(n-sr)!}{(n-mr)!(n-2sr+mr +m')!}&=&\frac{r^{2m'}(n-sr)!}{(n-sr+1)\cdots(n-mr)(n-2sr+mr +m')!}\\
&\leq &\frac{r^{2m'}(n-2sr+mr)!}{(n-2sr+mr+m')!}\;\;\leq\;\;\frac{r^{2m'}}{(n-2sr)^{m'}}.
\end{eqnarray*}
Thus, for $n$ large enough,
$$T_A\leq\left[\binom{n}{\underbrace{r,\ldots,r}_m}p^{m\binom{r}{2}}\right]^{-1}\frac{r^{2m'}}{(n/2)^{m'}}=\left[\binom{n}{\underbrace{r,\ldots,r}_m}p^{m\binom{r}{2}}\right]^{-1}\frac{2^{m'}r^{2m'}}{n^{m'}},$$
and the desired conclusion follows from Proposition~\ref{mainine} \red{if $m>0$, whereas the conclusion is obvious if $m=0$ for, then, $m'$ is positive. Note that Proposition~\ref{mainine} has to be applied for each possible value of $(m,m')$, but this is not a problem as there are at most $(s+1)^2$ such pairs.}
\end{proof}

Choose $\lambda$ with
$$0<\lambda<\frac{1}{1+2seq}$$
and consider the partition of $\{0,1,\ldots,r\}$ into the three sets
\begin{eqnarray*}
S_\lambda &=&\{x\in \mathbb{Z}\;\colon\hspace{.2mm}0\leq x \leq (1-\lambda)\log_q n\},\\
I_\lambda &=&\{x\in \mathbb{Z}\;\colon\hspace{.2mm}(1-\lambda)\log_q n< x < (1+\lambda)\log_q n\}, \mbox{ \ and}\\
L_\lambda &=&\{x\in \mathbb{Z}\;\colon\hspace{.2mm}(1+\lambda)\log_q n\leq x \leq r\}.
\end{eqnarray*}
An integer will be referred as \emph{short}, \emph{intermediate}, or \emph{large}, depending on whether it lies in $S_\lambda$, $I_\lambda$, or $L_\lambda$, respectively.

\medskip
Propositions~\ref{teorsmalllarge}--\ref{finalprop} below will enable us to bound from above each term in~(\ref{sumTA}) by a term $T_A$ for a suitable diagonal matrix $A$ as those in Proposition~\ref{matrix1r}.
 
 \begin{proposition}\label{teorsmalllarge}
\redd{There is a large integer $N$ (which depends only on the fixed parameters $s$, $\epsilon$, $p$, and $\lambda$) such that, for $n\geq N$:}
\begin{itemize}
  \item [i)] If $A'\in D$ arises by adding $1$ to a small entry in $A\in D$, then $T_{A'}<T_A$.
  \item [ii)] If $A'\in D$ arises by adding $1$ to a large entry in $A\in D$, then $T_A<T_{A'}$.
\end{itemize}
\end{proposition}
\begin{proof}
Suppose $A'\in D$ arises by increasing by $1$ an entry $a_{ij}=a$ in $A \in D$ (in particular $\Sigma A_i<r$ and $\Sigma A^j<r$). Then
$$\frac{T_{A'}}{T_A}=\frac{(r-\Sigma A^j)(r-\Sigma A_i) \, q^a}{(a+1)\left(n-2sr+\displaystyle{\sum_{k=1}^s} \Sigma A^k +1 \right)}$$
 Since $r=\red{o(n)}$, we have
$$
\frac{n}{2}\leq n-2sr+\displaystyle{\sum_{k=1}^s} \Sigma A^k +1 \leq n
$$
for $n$ large enough \redd{(depending only on $s$, $\epsilon$, and $p$), so that for those large values of $n$ we have}
 \begin{equation}\label{Bq^a}
  B q^a\leq\frac{T_{A'}}{T_A}\leq 2B q^a
 \end{equation}
 where
 \begin{equation}\label{B}
B=\frac{(r-\Sigma A^j)(r- \Sigma A_i)}{(a+1)n} \leq \frac{r^2}{n}.
\end{equation}

\smallskip\noindent
{\bf Case $a\in S_\lambda$}: We have
$q^a \leq q^{(1-\lambda)\log_q n}=n^{1-\lambda}$, so that
$$Bq^a\leq \frac{r^2n^{1-\lambda}}{n}=\frac{r^2}{n^{\lambda}}.$$
But $\displaystyle{\lim_{n\rightarrow\infty}} (r^2/n^{\lambda})=0$, so that the second inequality in~(\ref{Bq^a}) gives $T_{A'}<T_A$ for $n$ large enough \redd{(depending now on $s$, $\epsilon$, $p$, and $\lambda$).}

\medskip\noindent
{\bf Case $a\in L_\lambda$}: Now $q^a\geq n^{1+\lambda}$ \red{and, since $a+1\leq r<2\log_qn$ for $n$ large enough \redd{(depending only on $s$, $\epsilon$, and $p$), we have}}
$$B\geq\frac{1}{(a+1)n}\geq\frac{1}{2n\log_q n}.$$
Therefore
$$
B q^a\geq \frac{n^{1+\lambda}}{2n\log_q n}=\frac{n^{\lambda}}{2\log_q n}.
$$
This time the quotient $n^\lambda/(2\log_qn)$ tends to $\infty$ as $n\to\infty$, so that the first inequality in~(\ref{Bq^a}) gives $T_{A'}>T_A$ for $n$ large enough \redd{(depending now on $s$, $\epsilon$, $p$, and $\lambda$).}
\end{proof}

\begin{proposition}\label{propoinlar}
\redd{There is a large integer $N$ (which depends only on the fixed parameters $s$, $\epsilon$, $p$, and $\lambda$) such that, for $n\geq N$, the following assertion holds:} If $A'\in D$ arises by increasing by $1$ some entry $a_{ij}=a$ in $A\in D$ with $0<a \leq r/2$, then $T_A>T_{A'}$ for $n$ large enough provided the following two conditions hold:
\begin{enumerate}[(i)]
\item All small entries \red{in $A_i$ and in $A^j$} are zero.\vspace{-1mm}

\item There exists either an entry $a_{ij'}\neq0$ with $j'\neq j$, or an entry $a_{i'j}\neq0$ with $i'\neq i$.
\end{enumerate}
\end{proposition}
\begin{proof}
\red{Let $B$ be defined as in~(\ref{B}) so that~(\ref{Bq^a}) applies if $n$ is large enough. The fact that $r<2\log_qn$ (for large enough $n$'s) together with \emph{(i)} and \emph{(ii)} yield}
$$
(r-\Sigma A_i)(r-\Sigma A^j)\leq(r-\red{2}(1-\lambda) \log_q n)\,\red{r}\leq\red{(2\log_qn-\red{2}(1-\lambda) \log_q n) \, 2\log_qn}\leq\red{4} \lambda \log^2_q n.
$$
Moreover, since $a \leq r/2\leq\red{\log_qn-\log_q\log_qn+\log_q(e/2)+1}$, we have
$$
q^a\leq \frac{eqn}{2\log_qn}.
$$
Since $a+1\geq (1-\lambda)\log_q n$, the previous considerations amount to
$
2Bq^a\leq \frac{\red{4}\lambda e q}{1- \lambda}<1
$
where the last inequality comes from the definition of $\lambda$. The desired conclusion then follows from (\ref{Bq^a}).
\end{proof}

\begin{proposition}\label{finalprop}
\redd{There is a large integer $N$ (which depends only on the fixed parameters $s$, $\epsilon$, $p$, and $\lambda$) such that, for $n\geq N$:} If $A=(a_{i,j})$ is a matrix in $D$ with a non-zero entry $a_{ij}=a$ such that $a_{ij'}=a_{i'j}=0$ whenever $i\neq i'$ and $j\neq j'$, then $T_A \leq \max \{T_{A'}, T_{A''}\}$. Here $A'\in D$ is the matrix whose $(i,j)$-entry is $1$, whereas all its remaining entries agree with the corresponding entries of~$A$. Likewise, $A'' \in D$ is the matrix whose $(i,j)$-entry is $r$, whereas all its remaining entries agree with the corresponding entries of $A$.
\end{proposition}
\begin{proof}
We can assume $a\in I_\lambda\,$---otherwise the result follows by repeated used of Proposition~\ref{teorsmalllarge}. Let $B_\omega$ ($\omega=1,2$) arise by adding $\omega$ to the $(i, j)$-entry in $A$. Note that both $B_1$ and $B_2$ belong to $D$ if $n$ is large enough. Direct calculation yields
$$
\frac{T_{B_2}T_A}{T^2_{B_1}}=\left[\frac{(r-a-1)^2}{(r-a)^2}\cdot\frac{a+1}{a+2}\cdot\frac{n-2sr+\displaystyle{\sum_{k=1}^s} \Sigma A^k +1}{n-2sr+\displaystyle{\sum_{k=1}^s} \Sigma A^k+2}\right]\cdot q.
$$
\red{Each of the three quotients inside the bracket tends (uniformly on $a$) to 1 as $n\to\infty\,$---the first two because $a\in I_\lambda$. Since $q>1$, we get for large enough $n$ that $T_{B_2}T_A>T^2_{B_2}$ or, equivalently, that $\log_qT_A$ is a convex function on the interval $I_\lambda\,$---and even two units to the right of this open interval. The result now follows from Proposition~\ref{teorsmalllarge}.}
\end{proof}

Equation~(\ref{sumTA}) and, therefore, Theorem~\ref{randomTCs} now follow from Proposition~\ref{matrix1r}, Corollary~\ref{acotar} below, and the fact  that the size of $D-\{A_0\}$ increases (as $n\to\infty$) polynomially on $r$.

\begin{corollary}\label{acotar}
\redd{There is a large integer $N$ (which depends only on the fixed parameters $s$, $\epsilon$, $p$, and $\lambda$) such that, for $n\geq N$,} the term $T_A$ of any matrix $A \in D-\{A_0\}$ is bounded from above by a term $T_{A'}$ where $A'$ is a diagonal matrix in $D$ whose non-zero entries are either $1$ or $r$. \redd{(In general, the matrix $A'$ above depends on the given matrix $A$.)}
\end{corollary}

\begin{remark}\label{permutarcolumnas}{\em
Before proving Corollary~\ref{acotar}, it is useful to note that, from its bare definition, the term $T_A$ does not change after permuting the columns of $A\in D$. \red{In other words, if $A^\sigma$ is obtained by permuting the columns of $A\in D$ according to a permutation $\sigma$, then the rule $$\left(\rule{0mm}{4mm}(W_1,\ldots,W_s),(W'_1,\ldots,W'_s)\right)\mapsto \left(\rule{0mm}{4mm}(W_{\sigma(1)},\ldots,W_{\sigma(s)}),(W'_1,\ldots,W'_s)\right)$$ sets a 1-1 correspondence between pairs in $\mathcal{W}(s)$ with intersection type $A$, and pairs in $\mathcal{W}(s)$ with intersection type $A^\sigma$. In particular we can assume without lost of generality that the matrix $A\in D-\{A_0\}$ in Corollary~\ref{acotar} has non-zero entries on its main diagonal.}
}\end{remark}

\begin{proof}[Proof of Corollary~\ref{acotar}]
\redd{The following arguments hold for values of $n$ large enough so that Propositions~\ref{teorsmalllarge}--\ref{finalprop} apply.} If all entries in $A$ are short, then by Proposition~\ref{teorsmalllarge} there exists a diagonal matrix $A'\in D$, with zeros and ones on its main diagonal, satisfying $T_A \leq T_{A'}$. So we can assume that $A$ has at least one entry which is either intermediate or large, and that such an entry lies on the main diagonal. In addition, using again Proposition~\ref{teorsmalllarge}, we can assume that all small entries in $A$ are zero.

\smallskip
At this point, if on a given row (or column) of $A$ there are two non-zero entries, then Proposition~\ref{propoinlar} implies that one of them (the one which is at most $r/2$) can be lowered down to zero at the price of increasing the value of $T_A$---which is all right for the purposes of this proof. We can thus assume that each row (as well as each column) of $A$ has at most one non-zero entry. By Remark~\ref{permutarcolumnas}, this amounts to assuming that $A$ is a diagonal matrix. The proof is then completed by Proposition~\ref{finalprop}.
\end{proof}

\section{Higher topological complexity}\label{sectcs}
We prove Theorem~\ref{randomtopTCs} in this section. In slightly more detail, the results in the previous section are now used to study the behavior of the higher topological complexity of Eilenberg Mac-Lane spaces of type $(G, 1)$, for random right angled Artin groups $G$ with a large number of generators. Due to the role of right angled Artin groups in the theory of graph braid groups~(\cite{MR2077673}), our results become most relevant in the collision-free motion planning problem of a large number of particles on graphs.

\medskip
We start by reviewing the relevant definitions and constructions.

\medskip
For an integer $s\geq2$, the \emph{$s$-th higher} (also referred to as \emph{sequential}) \emph{topological complexity} of a path connected space $X$, $\TC_{s}(X)$, is defined by Rudyak in~\cite{Ru10} as the reduced Schwarz genus of the fibration $$e_s=e^X_{s}:X^{J_{s}}\to X^{s}.$$ Here $J_s$ is the wedge sum of $s$ (ordered) copies of the interval $[0,1]$, where $0\in[0,1]$ is the base point for the wedge, and $$e_s(f_1, \ldots, f_s)=(f_1(1), \ldots, f_s(1)), \quad \quad \mbox{for}  \quad  (f_1, \ldots, f_s) \in X^{J_s},$$ is the map evaluating at the extremes of each interval. Thus $\TC_{s}(X)+1$ is the smallest cardinality of open covers $\{U_i\}_i$ of $X^s$ so that $e_s$ admits a (continuous!) section $\sigma_i$ on each $U_i$. The elements of such an open cover, $U_i$, are called {\it local domains}, the corresponding sections $\sigma_i$ are called {\it local rules}, and the resulting family of pairs $\{(U_i,\sigma_i)\}$ is called a {\it motion planner}. We say that such a family is an {\it optimal motion planner} if it has $\TC_{s}(X)+1$ local domains. This number is a generalization of the concept of topological complexity introduced by Farber in \cite{Far} as a model to study the continuity instabilities in the motion planning of an autonomous system (robot) whose space of configurations is $X$. The term ``higher'' (or ``sequential'') comes from the consideration of a series of prescribed stages in the robot, and not only of initial-final stages as in Farber's original concept.

\medskip
The homotopy invariance of the $\TC$-concepts is a central feature that has captured much attention from topologists in recent years. In particular, standard obstruction theory can be used to obtain a general upper bound for $\TC_s(X)$ in terms of $\hdim(X)$, the homotopy dimension of $X\,$---that is, the minimal dimension of CW complexes having the homotopy type of $X$.

\begin{proposition}[{\cite[Theorem~3.9]{bgrt}}]\label{ulbTCn}
For a $c$-connected space $X$ with $c\geq0$, $$\TC_s(X)\leq s\hdim(X)/(c+1).$$
\end{proposition}

The topological spaces we are interested in arise as follows: Let $\Gamma$ be a graph with vertex set $[n]$ and flag complex $\Delta_{\Gamma}$, i.e.~the abstract simplicial complex whose $(k-1)$-simplices are the \red{cardinality-$k$} subsets of $[n]$ corresponding to complete subgraphs of $\Gamma$. The right angled Artin group
$$
G_{\Gamma}=\langle v \in V \,\, ; \,\, vw=wv \,\, \text{if and only if}\,\, \{v, w\}\, \text{is and edge of}\,\, \Gamma \rangle
$$
is closely related to $\Delta_\Gamma$. To spell out the connection, let $S^1=e^0\cup e^1$ be the 1-dimensional sphere with its minimal cell decomposition. Consider the $n$-dimensional torus, $T^n=(S^{1})^{\times n}=\bigcup e_J$, with its (also minimal) product cell decomposition, where the cells $e_J$, indexed by subsets $J\subset [n]$, are given by
$$
e_{J}=\{(x_1, \ldots, x_n) \in T^n\,\, | \,\,x_i=e^0 \,\text{if and only if}\,\, i \notin J \}.
$$
Let $K_{\Gamma}$ be the subcomplex of $T^n$ that results by deleting the cells indexed by subsets not corresponding to simplices in $\Delta_{\Gamma}$. In other words, $K_{\Gamma}$ is the polyhedral product space determined by $\Delta_\Gamma$ and the based circle $(S^1,e^0)$. As shown in~\cite[Theorem~10]{MR567067}, $K_{\Gamma}$ is an Eilenberg-MacLane complex of type $(G_{\Gamma},1)$, that is, $K_{\Gamma}$ is a path-connected space with fundamental group $G_\Gamma$ and trivial higher homotopy groups. 

\medskip
The relevance of Matula's Theorem~\ref{randomcat} for Theorem~\ref{randomtopTCs} can already be seen from Proposition~\ref{ulbTCn}: By definition, $K_\Gamma$ comes equipped with a CW structure having a $d$-dimensional cell for each complete subgraph of $\Gamma$ with $d$ vertices. In particular
\begin{equation}\label{hdimvsclique}
\hdim(K_\Gamma)\leq C(\Gamma)
\end{equation}
and, consequently,
\begin{equation}\label{easy-half}
\mathrm{Prob}(\TC_s\leq s\lfloor z+\epsilon\rfloor)\geq\mathrm{Prob}(C\leq\lfloor z+\epsilon\rfloor).
\end{equation}
As $n\to\infty$, the left hand side in~(\ref{easy-half}) tends to 1 since the right hand side does too in view of Matula's theorem. This gives half of Theorem~\ref{randomtopTCs}. Before proving the other half, namely the equality
\begin{equation}\label{otherhalf}
\lim_{n\to\infty}\mathrm{Prob}\left(s\lfloor z-\epsilon\rfloor\leq\TC_s\rule{0mm}{4mm}\right)=1,
\end{equation}
we pause to remark that~(\ref{hdimvsclique}) is in fact an equality, as follows easily from the following description of the cohomology ring of $K_\Gamma$. Recall that $H^*(T^n)$ (with any coefficients) is an exterior algebra $E(x_1,\ldots,x_n)$ where each generator has degree 1. Then:

\begin{proposition}[{See~\cite[Theorem 2.35]{MR2673742}}]\label{desccoho}
The inclusion $K_\Gamma\hookrightarrow T^n$ induces in cohomology an epimorphism $E(x_1,\ldots,x_n)\to H^*(K_\Gamma)$ whose kernel is additively generated by the monomials $\prod_{i\in J}x_i$ for which $e_J$ s not a cell of $K_\Gamma$. In particular, the cohomological dimension of $K_\Gamma$, $\cd(K_\Gamma)$, agrees with the Lusternik-Schnirelmann category of $K_\Gamma$, $\cat(\Gamma)$. Indeed, $$C(\Gamma)=\cd(K_\Gamma)\leq\cat(K_\Gamma)\leq\hdim(K_\Gamma)\leq C(\Gamma).$$
\end{proposition}

Proposition~\ref{desccoho} is used in~\cite{CosFar} in order to interpret the case $k=m=1$ of Proposition~\ref{mainine} as an indication that a significant amount of cohomology in $K_\Gamma$ is concentrated in dimension $\lfloor z-\epsilon\rfloor\,$, i.e.~within one of the top dimension of $K_\Gamma$. We leave to the reader the easy task of checking that the full form of Proposition~\ref{mainine}\hspace{.3mm}---our key technical input in Section~\ref{seccliqes}---asserts, more generally, that the expected number of $s$-th multi-cliques of size $\lfloor z-\epsilon\rfloor$ grows faster than any polynomial function on $\lfloor z-\epsilon\rfloor$. In topological terms, this implies that a significant amount of homology in any cartesian power $K_\Gamma^s$ is concentrated within $s$ units of the top dimension. The latter assertion should be compared to Corollary~\ref{neighborhood} below.

\medskip
We now explain how~(\ref{otherhalf}) follows from our previous work. By Proposition~\ref{desccoho}, the case $s=1$ reduces to Matula's Theorem~\ref{randomcat}. On the other hand, the case $s\geq2$ follows at once from Theorem~\ref{randomTCs} and the following result:

\begin{theo}[{\cite[Theorem~2.7]{GGY}}]\label{determdescri}For $s\geq2$,
$$
\TC_s(K_\Gamma)=\max\left\{\rule{0mm}{7mm}\sum_{\ell=1}^{s}\left|\rule{0mm}{3.3mm}V_\ell\right|-\left| \bigcap_{\ell=1}^{s}V_\ell\right| \colon \mbox{each } V_i\subseteq[n]\mbox{ yields a complete induced subgraph } \Gamma_{|V_i}\right\}.
$$
\end{theo}

We close the paper by noticing that the $s$-th higher topological complexity of $K_\Gamma$ is asymptotically almost surely within an $s$ neighborhood of the upper bound given in Proposition~\ref{ulbTCn}. Indeed, by Matula's Theorem~\ref{randomcat} (with $\epsilon<1/2$), the number $r=\lfloor z-\epsilon\rfloor$ in the previous section satisfies $r\geq\hdim-1$ asymptotically almost surely. So Theorems~\ref{randomTCs} and~\ref{determdescri} yield:

\begin{corollary}\label{neighborhood}
$\lim_{n\to\infty}\mathrm{Prob}\left(s(\hdim-1)\leq\TC_s \leq s\hdim\rule{0mm}{4mm}\right)=1.$
\end{corollary}


\bigskip\sc
Departamento de Matem\'aticas

Centro de Investigaci\'on y de Estudios Avanzados del IPN

Av.~IPN 2508, Zacatenco, M\'exico City 07000, M\'exico

{\tt jesus@math.cinvestav.mx}

{\tt
bgutierrez@math.cinvestav.mx}

{\tt rmas@math.cinvestav.mx}

\end{document}